\documentclass[a4paper, 11pt]{amsart}
\usepackage[latin1]{inputenc}
\usepackage[ T1]{fontenc}
\usepackage[english]{babel}
\usepackage{amssymb}
\usepackage{amsmath}
\usepackage{amsthm}
\usepackage{amscd}
\usepackage{amsfonts}
\usepackage{stmaryrd}
\usepackage{pb-diagram}
\usepackage{epic,eepic,epsfig}
\usepackage{a4wide}
\usepackage{nextpage}
\usepackage{fancyhdr}
\usepackage{enumerate}
\usepackage{hyperref}
\usepackage{float}
\usepackage{tikz}

\newtheorem{prop}{Proposition}[section]

\newtheorem{lemma}[prop]{Lemma}
\newtheorem{thm}[prop]{Theorem}

\theoremstyle{definition}
\newtheorem{rem}[prop]{Remark}

\newtheorem{definition}[prop]{Definition}

\newcommand{\PP}{{\mathbb{P}}}
\newcommand{\QQ}{{\mathbb{Q}}}

\newcommand{\degr} {\mathop{\mathrm{deg}}}

\usepackage{colonequals}

\title{Infinite extensions with finitely many CM moduli}
\author{Shu Kawaguchi and Fabien Pazuki}
\address{Department of Mathematics, Graduate School of Science, Kyoto University, Kyoto 606-8502, Japan.}
\email{kawaguch@math.kyoto-u.ac.jp}
\address{Department of Mathematical Sciences, University of Copenhagen,
Universitetsparken 5, 
2100 Copenhagen \O, Denmark.}
\email{fpazuki@math.ku.dk}
\date{March 2026}
\thanks{The authors thank the organizers of the \textit{Intercity seminar in Arakelov geometry} as part of the \textit{Arithmetic and algebraic geometry week in Iasi} in September 2025. FP thanks the organizers of the \textit{Intercity seminar in Arakelov geometry} in Kyoto, September 2024, and Kyoto University for the hospitality. FP thanks Ziyang Gao for a discussion around estimates on CM moduli in the non-primitive case. The authors thank Nuno Hultberg for his useful remarks and his question concerning uniformity in the dimension.}

\begin{document}

\maketitle

\begin{abstract}
    We show that there are uncountably many algebraic extensions of $\mathbb{Q}$ containing at most finitely many moduli of CM simple principally polarized abelian varieties of any fixed dimension $g\geqslant1$, generalizing a result of Hultberg in dimension 1.    
\end{abstract}

\smallskip

\noindent{\bf Keywords:} Abelian varieties, complex multiplication, heights, Northcott numbers.
    
\smallskip

\noindent{\bf 2020 Math. Subj. Classification:} 11G15, 11G50, 14G40.
    

\section{Introduction}

In recent work \cite[Theorem 5]{Hul26}, Hultberg proves that there exists uncountably many algebraic extensions of $\mathbb{Q}$ containing at most finitely many $j$-invariants of elliptic curves with complex multiplications (CM). The proof of this theorem is based on three inputs: estimates on the height of the $j$-invariant of CM elliptic curves in terms of the size of their endomorphism ring, following early work of Breuer \cite{Bre01}, estimates on the degree of definition of CM elliptic curves in terms of the size of their endomorphism ring \cite{Bre01}, and recent results on Northcott numbers of subsets of $\overline{\mathbb{Q}}$ associated with a weighted logarithmic Weil height \cite{OS23}.

For an abelian variety $A$ of higher dimension, the \textit{moduli} that we will use in place of the $j$-invariant comes from a classical construction (see for instance \cite{Mum66}): it is a projective point given by the image of the neutral element of $A$ by a theta embedding with theta structure of level $2$, see paragraph \ref{theta structure}. We will say that a field $L$ \textit{contains a moduli} when all the coordinates of the moduli are in $L$. The goal of this note is to prove the following generalization of Hultberg's result.

\begin{thm}\label{main theorem}
    Let $g\geqslant1$ be an integer. There exists uncountably many algebraic extensions of $\mathbb{Q}$ containing at most finitely many moduli of CM simple and principally polarized abelian varieties of dimension $g$.
\end{thm}

The proof of Theorem \ref{main theorem} is inspired by the one in dimension 1, however the moduli space of principally polarized abelian varieties has higher dimension, hence the needed estimates require more advanced technology. We use the following inputs: estimates on the Faltings height of CM abelian varieties in terms of the size of their endomorphism ring, following work of Tsimerman on the Andr\'e-Oort conjecture \cite{Tsi18}, height comparisons on the moduli space of principally polarized abelian varieties \cite{Paz12}, estimates on the degree of definition of the moduli of simple CM abelian varieties \cite{Tsi18}, and new results on Northcott numbers associated with a projective space, inspired by Northcott numbers of subsets of $\overline{\mathbb{Q}}$ associated with a weighted logarithmic Weil height \cite{OS23}. Let us add that the method of proof, like in dimension 1, also provides an explicit description of the type of fields obtained in Theorem \ref{main theorem}, as recalled in Remark \ref{explicit fields}.
An interesting question, suggested by Hultberg, is whether one may find infinite extensions of $\mathbb{Q}$ containing only finitely many CM moduli ranging over all dimensions.

\section{Definitions}
We start by gathering the necessary definitions of the various heights playing a role in the proof.

\subsection{Weil height and Northcott numbers}
Let $K$ be a number field. Let $M_K$ be a complete set of pairwise non-equivalent absolute values $\vert.\vert_v$ on $K$, given local degrees $d_v$ and normalized by $\vert p\vert_v=1/p$ for any $v\vert p$, for the underlying rational prime $p$. Let us also denote by $M_K^0$ the set of finite places and by $M_K^{\infty}$ the set of archimedean places. In the projective space of dimension $N\geqslant 1$, the logarithmic Weil height of a $K$-rational point $x=[x_0:\ldots:x_N]$ is given by the formula
\begin{equation}\label{proj Weil height}
h(x)=\frac{1}{[K:\mathbb{Q}]}\sum_{v\in{M_{K}}}d_v \log\max_{0\leqslant i \leqslant N}\{\vert x_i\vert_v\},
\end{equation}
and in particular the height of an algebraic number $\alpha$ in $K$ is defined by $h(\alpha)=h([1:\alpha])$, with explicit formula 
\begin{equation}\label{Weil height}
h(\alpha)=\frac{1}{[K:\mathbb{Q}]}\sum_{v\in{M_{K}}}d_v \log\max\{1, \vert \alpha\vert_v\}.
\end{equation}
By the classical extension formula, $h(\cdot)$ is in fact a height well defined on $\mathbb{P}^N(\overline{\mathbb{Q}})$. We also define the following weighted logarithmic Weil height: for any $\gamma\in{\mathbb{R}}$, let us define a function $h_\gamma$ for $x=[1:x_1:\ldots:x_N]$ in $\mathbb{P}^N(\overline{\mathbb{Q}})$ by the formula 
\begin{equation}\label{height gamma}
h_\gamma(x)=\deg(x)^{\gamma} \, h(x),
\end{equation}
where $\deg(x)=[\mathbb{Q}(x_1,\ldots, x_N):\mathbb{Q}]$. In particular, for any algebraic number $\alpha$, we have $h_\gamma(\alpha)=\deg(\alpha)^\gamma h(\alpha)$. Another useful projective height of $x=[x_0:\ldots:x_N]$ features an $L^2$ variant at archimedean places: 
\begin{equation}\label{proj L^2 Weil height}
h_{(2)}(x)=\frac{1}{[K:\mathbb{Q}]}\sum_{v\in{M_{K}^0}}d_v \log\max_{0\leqslant i \leqslant N}\{\vert x_i\vert_v\}+\sum_{v\in{M_{K}^\infty}}d_v \log\left(\sum_{0\leqslant i \leqslant N}\vert x_i\vert_v^2\right)^\frac{1}{2}.
\end{equation}

Let us recall the definition of Northcott numbers from \cite{PTW22}. Note that the Northcott number with respect to the Weil height was introduced by Vidaux and Videla \cite{VV16}, and refines the concept of the Northcott property given by Bombieri and Zannier in \cite{BZ01}.

\begin{definition}[Northcott number]\label{Northcottnum}
For a subset $S$ of  the algebraic numbers $\overline{\mathbb{Q}}$ and a function $f:\overline{\mathbb{Q}}\to [0,\infty)$
we set 
$$\mathcal{N}_f(S)=\inf\{t \in [0, \infty); \#\{\alpha\in S; f(\alpha)< t\}=\infty\},$$
with  the usual interpretation $\inf \emptyset=\infty$. 
We call $\mathcal{N}_f(S)\in [0,\infty]$ the \emph{Northcott number} of $S$ with respect to $f$.
If $\mathcal{N}_f(S)=\infty$ then we say that $S$ has the Northcott property with respect to $f$. 
\end{definition}

We will also need a Northcott number attached to projective spaces of dimension $N\geqslant1$. For a field $L \subset \overline{\QQ}$ and a function $f:\mathbb{P}^N(\overline{\mathbb{Q}})\to [0,\infty)$, we set 
\begin{equation}\label{Northcott projective space}
\mathcal{N}_{f}(\PP^N(L)) = 
\inf \{t \in [0, \infty) ; 
\#\{
x \in \PP^N(L) ; f(x) < t
\}
= \infty
\}. 
\end{equation}

\subsection{Faltings height}
Let $A$ be an abelian variety defined over $\overline{\mathbb{Q}}$, of dimension $g\geqslant 1$, and
$K$ a number field over which $A$ is rational and semi-stable. Let $\pi\colon {\mathcal A}\longrightarrow \mathcal{S}$ be a semi-stable model of $A$ over $\mathcal{S}=\mathrm{Spec}(\mathcal{O}_K)$, where $\mathcal{O}_K$ is the ring of integers of $K$. We shall denote by $\varepsilon$ the zero section of
$\pi$ and by
$\omega_{{\mathcal A}/\mathcal{S}}$ the sheaf of maximal exterior
powers of the sheaf of relative differentials
$$\omega_{{\mathcal A}/\mathcal{S}}:=\varepsilon^{\star}\Omega^g_{{\mathcal
A}/\mathcal{S}}\simeq\pi_{\star}\Omega^g_{{\mathcal A}/\mathcal{S}}\;.$$

For any embedding $\sigma$ of $K$ in $\mathbb{C}$, the corresponding line bundle
$$\omega_{{\mathcal A}/\mathcal{S},\sigma}=\omega_{{\mathcal A}/\mathcal{S}}\otimes_{{\mathcal O}_K,\sigma}\mathbb{C}\simeq H^0({\mathcal
A}_{\sigma}(\mathbb{C}),\Omega^g_{{\mathcal A}_\sigma}(\mathbb{C}))$$
can be equipped with the metric $\Vert.\Vert_{\sigma}$ defined by
$$\Vert\alpha\Vert_{\sigma}^2=\frac{i^{g^2}}{(2\pi)^{2g}}\int_{{\mathcal
A}_{\sigma}(\mathbb{C})}\alpha\wedge\overline{\alpha},$$
where we follow a normalization which will ensure that the height defined below is non-negative, as done in \cite[D\'efinition 3.2]{Paz19}. 

The ${\mathcal O}_K$-module of rank one $\omega_{{\mathcal A}/\mathcal{S}}$, together with the hermitian norms
$\Vert.\Vert_{\sigma}$ at infinity defines a hermitian line bundle 
$\overline{\omega}_{{\mathcal A}/\mathcal{S}}$ over $\mathcal{S}$, which has a well defined  Arakelov degree
$\widehat{\degr}(\overline{\omega}_{{\mathcal A}/\mathcal{S}})$, given by
$$\widehat{\degr}(\overline{\omega}_{{\mathcal A}/\mathcal{S}})=\log\#\left({\omega}_{{\mathcal A}/\mathcal{S}}/{{\mathcal
O}}_Ks\right)-\sum_{\sigma\colon K\hookrightarrow \mathbb{C}}\log\Vert
s\Vert_{\sigma}\;,$$
where $s$ is any non-zero element of $\omega_{{\mathcal A}/S}$. This formula does not depend on the choice
of $s$ in view of the product formula on $K$. 

We now give the definition of the Faltings height that one finds in \cite[D\'efinition 3.2]{Paz19} page 127, which is a simple translate by a quantity depending only on $g$ of the original one in \cite{Falt} page 354. 

\begin{definition}\label{FaltHeight}  The normalized stable Faltings height of $A$ is defined as
$$h(A):=\frac{1}{[K:\mathbb{Q}]}\widehat{\degr}(\overline{\omega}_{{\mathcal
A}/\mathcal{S}})\;.$$
\end{definition}
It is a real number which is stable by field extension, as $A$ is semi-stable over $K$. It is moreover non-negative thanks to an inequality of Bost \cite{Bost}, and the proof details are accessible in \cite{GauR}.

\subsection{Theta height of level 2}\label{theta structure}

Given an abelian variety $A$ defined over $\overline{\mathbb{Q}}$ principally polarized by an ample line bundle $\mathcal{L}$,
the line bundle $\mathcal{L}^{\otimes 4}$ is very ample on $A$. Therefore, it
defines an embedding
\begin{equation}\label{thetaplonge}
\Theta\colon A\longrightarrow \mathbb{P}^{4^{g}-1}.
\end{equation}
The theta height of $A$ with respect to $\mathcal{L}$ is then defined as in \cite[Definition 2.6]{Paz12} page 30, with the choice of theta structure of level $r=2$:
$$h_{\Theta}(A,\mathcal{L}):=h_{(2)}(\Theta(0)),$$
where $h_{(2)}(\cdot)$ is the projective Weil height on $\mathbb{P}^{4^{g}-1}$ defined by (\ref{proj L^2 Weil height}).

\section{Height and degree inequalities}

In this section we gather the results needed for the proof of Theorem \ref{main theorem}. We start with an estimate on the degree of the field of moduli of a simple abelian variety $A$ with complex multiplication (CM) given in Theorem \ref{Tsim1}, and an estimate on the Faltings height of any CM abelian variety given in Theorem \ref{Tsim2}. These results are both coming from the work of Tsimerman \cite{Tsi18}. Then we recall how to compare the Faltings height with the theta height in Theorem \ref{theta-faltings}, as was done in \cite{Paz12} following ideas of Bost and David. Finally we move on to prove in Lemma \ref{finiteness} and Lemma \ref{northcott:proj} that fields with big Northcott numbers contain the moduli of finitely many abelian varieties with bounded height, if one picks the correct height function, and the correct upper bound.

Let us focus on CM abelian varieties. Recall that for a CM field $E$ of degree $2g$, a type of $E$ is a set $\Phi$ of $g$ complex embeddings $E\hookrightarrow \mathbb{C}$ such that the set of all $2g$ complex embeddings of $E$ is the disjoint union of $\Phi$ and its complex conjugate $\overline{\Phi}$. We denote the discriminant of $E$ by $\mathrm{Disc}(E)$. A primitive type of $E$ is a type that doesn't come from a strict CM subfield of $E$. An abelian variety is said to be CM by $E$ if its ring of endomorphisms contains the full ring of integers $\mathcal{O}_E$, which is often called \textit{maximal} CM. 

We denote by $\mathbb{Q}(A)$ the field of moduli of the abelian variety $A$. It is a number field, see for instance the classical reference \cite{ST61}, Proposition 26 page 109. Let us recall Tsimerman's \cite{Tsi18}, Theorem 4.2 page 385.

\begin{thm}\label{Tsim1}
There exist $\delta(g)>0$ and $c_1(g)>0$ such that for any CM field $E$ of degree $2g$, any primitive type $\Phi$ for $E$, and any abelian variety $A$ which is CM by $E$ with type $\Phi$, the field of moduli $\mathbb{Q}(A)$ satisfies $$[\mathbb{Q}(A):\mathbb{Q}] \geqslant c_1(g) \vert \mathrm{Disc}(E)\vert^{\delta(g)}.$$ 
\end{thm}

\begin{rem}
    Since $\Phi$ is primitive, an abelian variety $A$ in Theorem~\ref{Tsim1} is simple (see for instance \cite[Lemma 1 page 508]{Shi71}).
    If $A$ is not simple, Theorem 5.2 of \cite{Tsi18} and Conjecture 7.1 of \cite{Tsi12} give a lower bound in terms of the discriminant of the center of the endomorphism ring. Note that in general, the field of definition of an algebraic variety may be bigger than its field of moduli, see for instance \cite[Theorem 6 page 370]{Mat56}. Fields of moduli of CM abelian varieties are treated in \cite{Shi71}. Shimura also proved in \cite{Shi72} that a generic abelian variety of odd dimension can always be defined over its field of moduli, whereas a generic abelian variety of even dimension can never be defined over its field of moduli. See \cite{Mes91} for a concrete treatment in the case of curves of genus 2.
\end{rem}

Let us also recall Tsimerman's \cite{Tsi18}, Corollary 3.3 page 383.

\begin{prop}\label{Tsim2}
Let $g\geqslant1$ be an integer and 
let $\epsilon > 0$ be a positive number. Then there exists $c_2(g, \epsilon) > 0$ with the following properties. Let $E$ be a CM field of degree $2g$, let $\Phi$ be a CM type and let $A$ be an abelian variety admitting CM by $E$ with type $\Phi$. Then 
$$h(A)\leqslant c_2(g, \epsilon)\, \vert \mathrm{Disc}(E) \vert^{\epsilon}.$$ 
\end{prop}

Let us state the explicit comparison (see Corollary 1.3 pages 21-22 of \cite{Paz12}) between the theta height and the Faltings height of an abelian variety.
\begin{thm}\label{theta-faltings}
    Let $g\geqslant1$ be an integer. There exists $c_3(g)>0$ such that, for any abelian variety $A$ defined over $\overline{\mathbb{Q}}$, of dimension $g$, equipped with a principal polarization associated with a line bundle $\mathcal{L}$, and with level $2$ theta structure, if we denote by $h_\Theta=\max\{1,h_{\Theta}(A,\mathcal{L})\}$ its theta height and by $h_F=\max\{1,h(A)\}$ its stable Faltings height, then one has 
    $$\big\vert h_\Theta-\frac{1}{2}h_F\big\vert\leqslant c_3(g)\log(\min\{h_\Theta, h_F\}+2). $$
\end{thm}

Let us now display the link between finiteness of the set of moduli and large Northcott numbers for a well-chosen height. 

\begin{lemma}\label{finiteness}
    Let $\gamma < 0 $ and $C>0$. Let $L$ be an algebraic extension of $\mathbb{Q}$ with Northcott number $\mathcal{N}_{h_\gamma}(\mathbb{P}^{4^g-1}(L))\geqslant C$. Then there are at most finitely many abelian varieties of dimension $g$ and level $2$ theta structure, with moduli defined over $L$, and such that $[\mathbb{Q}(A):\mathbb{Q}]^\gamma h(A)\leqslant C'$, as soon as $C>(1+2c_3(g))(1+C')$.
\end{lemma}

\begin{proof}
    By Theorem \ref{theta-faltings}, we know that $$\max\{1,h(A)\}\geqslant 2 \max\{1,h_\Theta(A,\mathcal{L})\}-2 c_3(g)\log(\max\{1,h(A)\}+2),$$ which leads to the weaker $$(1+2c_3(g))\max\{1,h(A)\}\geqslant  \max\{1,h_\Theta(A,\mathcal{L})\},$$ which in turn gives the inequality $$(1+2c_3(g))(1+C')\geqslant [\mathbb{Q}(A):\mathbb{Q}]^\gamma h_\Theta(A,\mathcal{L}),$$ and there are at most finitely many $(A,\mathcal{L})$ with moduli defined over the field $L$, up to isomorphisms, satisfying this inequality: indeed, the moduli point is given by projective coordinates $\Theta(0)=[a_0:\cdots:a_N]$ for $N=4^g-1$, where we may assume that $a_0=1$ (up to re-indexing if necessary), and we have $h_{(2)}([1:a_1: \cdots: a_N]) \geqslant h([1:a_1: \cdots: a_N])$, where $h(\cdot)$ is defined in (\ref{proj Weil height}). Moreover $\gamma < 0$, so we have $$[\mathbb{Q}(A):\mathbb{Q}]^\gamma\, h_\Theta(A,\mathcal{L}) \geqslant 
[\QQ(a_1, \ldots,  a_N):\mathbb{Q}]^\gamma\, h([1:a_1: \cdots: a_N])
$$ and by assumption $$\mathcal{N}_{h_\gamma}(\mathbb{P}^{4^g-1}(L))\geqslant C>(1+2c_3(g))(1+C'),$$ so we conclude on the finiteness by the definition of the Northcott number.
\end{proof}

We now need to prove the existence of algebraic extensions $L$ such that the Northcott number $\mathcal{N}_{h_\gamma}(\mathbb{P}^{4^g-1}(L))$ is large. Let us start with the following statement, which is a particular case of Theorem 1.4 page 129 of \cite{OS23}.

\begin{thm}\label{fields}
    Let $f=h_\gamma$ for $\gamma < 0$, and let $C>0$ be an absolute constant. There are uncountably many fields $L$ such that $\mathcal{N}_{f}(L)\geqslant C$.
\end{thm}

Let us add a remark on the type of fields used to prove Theorem \ref{fields}.

\begin{rem}\label{explicit fields}
    In Theorem 4.1 page 135 from \cite{OS23}, the fields are constructed explicitly: they are infinite extensions of the shape $$\mathbb{Q}\left(\Big(\frac{p_i}{q_i}\Big)^{\frac{1}{d_i}}_{i\in{\mathbb{N}}}\right),$$ for some sequences of primes $(p_i)_{i\in{\mathbb{N}}}$, $(q_i)_{i\in{\mathbb{N}}}$, $(d_i)_{i\in{\mathbb{N}}}$ satisfying specific properties, generalizing the type of fields constructed in \cite{PTW22}. With the notation on page 133 in \cite{OS23}, $K_0 = \QQ$ and, 
for any $i\geqslant1$, for any $a \in K_i \setminus K_{i-1}$, we have 
$h_\gamma(a) \geqslant C - \frac{\log(d_i)}{2 (d_1 \cdots d_i)^\gamma(d_i - 1)}$, which is large if $C$ is large. In particular, 
for any $a \in L \setminus \QQ$, we may request 
$h_\gamma(a)\neq0$. Further, 
for any $a \in L$, if $h_\gamma(a)\neq0$, we may request $h_\gamma(a) \geqslant \log2$. By Hultberg \cite{Hul26}, one may replace the base $K_0$ by a finite extension of $\mathbb{Q}$ if needed.
\end{rem}

\begin{lemma}
\label{northcott:proj}
Let $\gamma < 0$ and let $C \geqslant \log 2$ be an absolute constant. There exist uncountably many fields $L \subset \overline{\QQ}$ such that $\mathcal{N}_{h_{N \gamma}}(L) \geqslant C$ and 
 $\mathcal{N}_{h_{\gamma}}(\PP^N(L)) \geqslant (\log2)^{1-\frac{1}{N}} \cdot C^{\frac{1}{N}}$. 
\end{lemma}

\begin{proof}
Let $L$ be a field as constructed in \cite{OS23} for $h_{N \gamma}$, with the extra request from Remark \ref{explicit fields}. Take an element $\alpha = [x_0: x_1: \cdots: x_N] \in \PP^N(L)$. We may assume (up to re-indexing) that $x_0 \neq 0$, and write $\alpha = [1: a_1: \cdots: a_N]$. 
Set $I = \{i \mid h_{N \gamma}(a_i) \neq 0\} \subseteq \{1, 2, \ldots, N\}$. 
With the request from Remark \ref{explicit fields}, 
if $[\QQ(a_i): \QQ] \geqslant 2$, then $i  \in I$.
Since $\gamma < 0$, we have 

\begin{align*}
& [\QQ(a_1, \ldots,  a_N):\mathbb{Q}]^{\gamma}\, h([1:a_1: \cdots: a_N]) 
\geqslant 
\left(\prod_{i \in I} [\QQ(a_i): \QQ]^{\gamma}\right)
\, 
\left( \prod_{i \in I} h([1:a_i])\right)^{\frac{1}{\#I}}
\\
& \quad 
\geqslant 
\left(\prod_{i \in I} [\QQ(a_i): \QQ]^{N \gamma}\, h([1:a_i])\right)^{\frac{1}{\#I}}
= \left(\prod_{i \in I} h_{N \gamma}(a_i)\right)^{\frac{1}{\#I}}. 
\end{align*}

Let $C_1 = (\log2)^{1-\frac{1}{N}} \cdot C^{\frac{1}{N}}$ denote the lower bound to reach in order to prove the statement, and 
assume that 
$C_1 > [\QQ(a_1, \ldots,  a_N):\mathbb{Q}]^\gamma\, h([1:a_1: \cdots: a_N])$. 
Since $h_{N \gamma}(a_j) \geqslant \log2$ for any $j \in I$, we have 
\[
h_{N \gamma}(a_i) < C_1^{\#I} \prod_{j \in I, j \neq i} h_{N \gamma}(a_i)^{-1} \leqslant 
C_1^{\#I} ((\log2)^{-1})^{\#I-1} \leqslant (C_1\, (\log2)^{-1})^{N} \cdot \log2 
= C. 
\]
Since $\mathcal{N}_{h_{N \gamma}}(L) \geqslant C$, there are only finitely many such $a_1, \ldots, a_N$. From this we conclude $\mathcal{N}_{h_{\gamma}}(\PP^N(L)) \geqslant  C_1$. 
\end{proof}

\section{Proof of Theorem \ref{main theorem}}

We now have the tools to prove Theorem \ref{main theorem}. Let us fix $g\geqslant1$. Let $E$ be a CM field of degree $2g$, let $A$ be a principally polarized abelian variety admitting CM by $E$ with primitive type $\Phi$. By Theorem \ref{Tsim1} and Proposition \ref{Tsim2}, taking $\epsilon = \frac{\delta(g)}{2}$, we have 
\begin{equation}\label{final}
[\mathbb{Q}(A):\mathbb{Q}]^{-1}h(A)\leqslant c_1(g)^{-1} c_2\left(g, \frac{\delta(g)}{2}\right) \vert \mathrm{Disc}(E)\vert^{-\frac{\delta(g)}{2}}
\leqslant C^\prime, 
\end{equation}
where $C^\prime := c_1(g)^{-1} c_2\left(g, \frac{\delta(g)}{2}\right)$ is a constant depending only on $g$. 
By Lemma \ref{finiteness} we get finitely many CM moduli defined over any field $L$ such that the Northcott number $\mathcal{N}_{h_\gamma}(\mathbb{P}^N(L))\geqslant C_1>0$ with $\gamma=-1$ and $C_1$ large enough depending on $g$. There exists uncountably many such fields thanks to Lemma \ref{northcott:proj}.

\end{document}